\documentclass[11pt,reqno]{amsart}

\usepackage{amsthm,amsmath,amssymb}
\usepackage{graphicx}
\usepackage{float}
\usepackage{enumitem}
\usepackage[colorlinks=true,
linkcolor=blue,
urlcolor=blue,
citecolor=blue]{hyperref}
\usepackage{mathtools}
\usepackage{relsize}
\usepackage{tikz}
\usepackage[toc,page]{appendix}
\usepackage{pdflscape}
\usepackage{multirow}
\usepackage{adjustbox,expl3,etoolbox} 
\usepackage[margin = 3cm]{geometry}
\usepackage{mathrsfs}
\usepackage{rotating}
\usepackage{tablefootnote}
\usepackage{longtable}
\usepackage{comment}
\usepackage{xcolor}
\usepackage{stmaryrd}
\usepackage{url}

\theoremstyle{definition}
\newtheorem{definition}{Definition}[section]

\newcommand{\F}{\mathbb{F}}
\newcommand{\Q}{\mathbb{Q}}
\newcommand{\C}{\mathbb{C}}
\newcommand{\N}{\mathbb{N}}
\newcommand{\Z}{\mathbb{Z}}

\newcommand{\R}{\mathbb{R}}
\newcommand{\HH}{\mathbb{H}}
\newcommand{\st}{{|}}

\newtheorem{theorem}{Theorem}[section]

\newtheorem{remark}[theorem]{Remark}
\newtheorem{proposition}[theorem]{Proposition}

\newtheorem{lemma}[theorem]{Lemma}

\newcommand{\ord}{{\rm ord}}

\newcommand{\id}{{\rm id}}

\title{On a problem of Amitsur and Small}
\author{Elad Paran}
\date\today

\begin{document}


\begin{abstract} We construct an example of a division ring $D$ and a maximal left ideal $M$ in the polynomial ring $D[x,y]$ in two central variables over $D$, such that the intersection $M \cap D[x]$ is not a maximal left ideal in $D[x]$. This resolves a ring-theoretic problem of Amitsur and Small raised in 1978.\end{abstract}

\maketitle


%

\section{Introduction}

Let $D$ be a division ring and let $M$ be a maximal left ideal in the ring $D[x_1,\ldots,x_n]$ of polynomials in $n$ central variables over $D$. In \cite{AmitsurSmall1978}, Amitsur and Small prove a form of Nullstellensatz for such ideals: The left quotient module $D[x_1,\ldots,x_n]/M$ is finite over $D$. A key lemma in their proof states that the intersection of $M$ and $D[x_k]$ is non-zero, for any $1 \leq k \leq n$. In the case where $D$ is a field, one can of course say more: The intersection $M \cap D[x_k]$ is a maximal ideal. This is a classical fact, essential to dimension theory. However, in the case where $D$ is an arbitrary division ring, Amitsur and Small write ``We remark that we are unable to show that maximal left ideals in $D[x_1,\ldots,x_n]$ intersect $D[x_1,\ldots,x_k]$, $k<n$, in maximal or even semi-maximal left ideals" \cite[p.~356]{AmitsurSmall1978}. The second part of this problem was solved by Small and Robson in \cite[Proposition 5.3]{RobsonSmall1981}, who showed that if $M$ is a semi-maximal left ideal in $D[x_1,\ldots,x_n]$ then $M \cap D[x_1,\ldots,x_k]$ is semi-maximal for any $1 \leq k \leq n$ (see also \cite[Theorem 6.8, p.~360]{RobsonMc}). However, the first part of the problem (whether maximal ideals contract to maximal ideals) remained unresolved by the mention works. Some related results are given by Resco in \cite{Resco1980}, who notes that ``It remains unknown, however, whether a maximal right ideal need contract to a maximal right ideal" \cite[p.~70]{Resco1980}\footnote{Some of the mentioned papers work with left ideals, and some with right ideals, but of course, the two notions are interchangeable.}. Years later, in a paper of Rowen from 1995, he mentions \cite[p.~2272]{Rowen1995} that the solution to this question is negative ``as shown recently by Amitsur and Small". However, no reference is given by Rowen, and there does not seem to be a paper from that period containing a solution. The author had written to Small, who does not recall that he or Amitsur had solved this problem, and to Rowen, who also does not recall what the mentioned sentence in his paper refers to. (Notably, Amitsur had died at September 1994, the same month when Rowen's paper was submitted for publication.) 

In this note we resolve the problem of Amitsur and Small (or re-discover its solution, had Amitsur and Small already found it), by providing an example of a division ring $D$ and a maximal left ideal $M$ in $D[x,y]$, such that $M \cap D[x]$ is not maximal. This example is described in \S\ref{sec:counter} below, where a simple recipe for constructing such examples is given by Proposition \ref{main}. Then in \S\ref{null} we discuss some connections of this problem to recent works, leading to several natural follow-up problems.



\section*{Acknowledgement} The author wishes to thank Lance Small and Louis Rowen for their input regarding the topic of this paper. 

\section{Maximal left ideals that contract to non-maximal left ideals}\label{sec:counter}

Throughout this note, given a division ring $D$, we denote by $D[x_1,\ldots,x_n]$ the ring of polynomials in $n$ central variables over $D$, and in the case of one or two variables, we shall write $D[x]$ or $D[x,y]$, respectively. Given a polynomial $f = f_nx^n+f_{n-1}x^{n-1}+\ldots+f_1x+f_0 \in D[x]$ and an element $a \in D$, we denote by $f(a)$ the usual substitution $f(a) = f_na^n+f_{n-1}a^{n-1}+\ldots+f_1a+f_0$. It is well-known that the substitution map $f \mapsto f(a)$ is generally not a homomorphism from $D[x]$ to $D$, however it satisfies the following ``product formula": If $f,g \in D[x]$ then $(fg)(a) = 0$ if $g(a) = 0$, and $(fg)(a) = f\big(a^{g(a)}\big)g(a)$ if $g(a) \neq 0$ , where $a^{g(a)}$ denotes the conjugation $g(a)ag(a)^{-1}$ (see \cite[Proposition 16.3, p.~263]{Lam91}). 

We say that $a \in D$ is a zero of $f \in D[x]$ if $f(a) =0$. This is equivalent to $x-a$ being a right-hand factor of $f$ \cite[Proposition 16.2, p.~262]{Lam91}.

We shall need the following lemma:

\begin{lemma}\label{proper}
    Let $a,b,c$ be elements of a division ring $D$ such that $(ab)c=c(ab)$ and $(a+b)c=c(a+b)$. Then the left ideal in $D[x,y]$ generated by $(x-a)(x-b),y-c$ is a proper left ideal. 
\end{lemma}
\begin{proof}
    Suppose to the contrary that there exist $f,g \in D[x,y]$ such that $$f\cdot (x-a)(x-b)+g \cdot (y-c) =1.$$
The polynomial ring $D[x]$ is an Ore domain, hence admits a unique quotient division ring, which we denote by $D(x)$. Let us view the above presented equality as taking place in the ring $D(x)[y]$. Consider the substitution map $y \mapsto c$ from $D(x)[y]$ to $D(x)$. Applying this map to the presented equality, we get, using the substitution formula mentioned above, that $$f(c^{(x-a)(x-b)}) \cdot (x-a)(x-b) = 1.$$
    (Let us emphasize that in this equality $(x-a)(x-b)$ serves as a ``constant" -- an element of the division ring $D(x)$.) Now, by our assumptions, $c$ commutes with $a+b$ and with $ab$, hence $c$ commutes with $(x-a)(x-b)=x^2-(a+b)x+ab$. Thus in the ring $D(x)$, we have $c^{(x-a)(x-b)} = c$. But this means that $f(c^{(x-a)(x-b)})$, which a priori is an element of $D(x)$, equals $f(c)$, which belongs to $D[x]$. This contradicts the last presented equality, since then $f(c^{(x-a)(x-b)}) \cdot (x-a)(x-b)=f(c)\cdot (x-a)(x-b) \in D[x]$ is either $0$ or of degree at least $2$ in $x$. 
\end{proof}

Given a division ring $D$ and $n \geq 1$, we shall denote by $D_c^n$ the subspace 

$$D_c = \{(a_1,\ldots,a_n) \in D^n \st a_ia_j = a_ja_i \hbox{ for all } 1\leq i,j \leq n\}$$
of $D^n$. 

\begin{lemma}\label{whole}
Let $a_1,\ldots,a_n$ be elements of a division ring $D$. Then the left ideal in $D[x_1,\ldots,x_n]$ generated by $x_1-a_1,\ldots,x_n-a_n$ is a proper left ideal if and only if $(a_1,\ldots,a_n)\in D_c^n$. Moreover, if this left ideal is proper, then it is maximal.
\end{lemma}
\begin{proof}
The first claim is given by in \cite[Lemma 2.1]{AP20}, and the second claim is given by \cite[Proposition 2.2]{AP20} -- both claims are formulated there for the case where $D$ is the real quaternion algebra, but their proofs hold verbatim for any division ring.
\end{proof}

We remark that by the above lemma, the points of $D_c^n$ are precisely the points for which substitution in polynomials in $D[x_1,\ldots,x_n]$ is well-defined, as discussed in \cite[\S2]{AP20}.

Next, we have:

\begin{proposition}\label{main}
    Let $D$ be a division ring and let $a,b,c$ be elements of $D$ satisfying the following conditions:

\begin{enumerate}[label=\alph*)]
    \item $abc = cab$.
    \item $c(a+b) = (a+b)c$.
    \item Every zero in $D$ of the polynomial $(x-a)(x-b)$ does not commute with $c$.

\end{enumerate}

Then the left ideal $M$ generated by $(x-a)(x-b),y-c$ in $D[x,y]$ is a maximal left ideal, whose intersection with $D[x]$ is not a maximal left ideal in $D[x]$.
\end{proposition}
\begin{proof}
 By the preceding lemma $M$ is a proper left ideal in $D[x,y]$. To show that it is maximal, suppose that $f \in D[x,y]$ does not belong to $M$. We must show that $M+D[x,y]f = D[x,y]$.  

By left-hand division with remainder by $y-c$ in the ring $D[x,y] = D[x][y]$, we may write $f = g \cdot (y-c)+h$, for some $g \in D[x,y]$ and $h \in D[x]$. Similarly, by division with remainder in $D[x]$ by $(x-a)(x-b)$, we may write $h = p(x-a)(x-b)+ux-v$ for some $u,v \in D$. Then it suffices to prove that $M+D[x,y](ux-v) = D[x,y]$. Suppose first that $u = 0$. Then if $v = 0$ we get that $f = g\cdot(y-c)+p\cdot(x-a)(x-b)$ belongs to $M$, a contradiction, and if $v \neq 0$ we have $$M+D(x,y)(ux-v) \supseteq D[x,y](ux-v) = D[x,y]v = D[x,y].$$

Next, suppose that $u \neq 0$. Then we may assume without loss of generality that $u = 1$ (by replacing $ux-v$ with $ux-u^{-1}v$, if necessary). If $v$ is a zero of $(x-a)(x-b)$, then $x-v$ is a right-hand factor of $(x-a)(x-b)$, hence$$M+D[x,y](x-v) = \big(D[x,y](y-c)+D[x,y](x-a)(x-b)\big)+D[x,y](x-v)$$ $$ = D[x,y](y-c)+\big(D[x,y](x-a)(x-b)+D[x,y](x-v)\big)$$ $$=D[x,y](y-c)+D[x,y](x-v).$$
However, by our assumptions we have $vc \neq vc$, hence by Lemma \ref{whole} the presented ideal is the whole ring $D[x,y]$, as needed. Next, if $v$ is not a zero of $(x-a)(x-b)$, then $x-v$ is coprime to $(x-a)(x-b)$ in $D[x]$, which means that $1 \in D[x](x-a)(x-b)+D[x](x-v) \subseteq M+D[x,y](x-v)$, hence $M+D[x,y](x-v) = D[x,y]$. We conclude that $M$ is a maximal left ideal in $D[x,y]$.

Now, note that the intersection of $M$ with $D[x]$ is generated, as a left ideal, by $(x-a)(x-b)$. Indeed, since $D[x]$ is a principal left ideal domain and $(x-a)(x-b) \in M$, we must have $M \cap D[x] = D[x]p$, for some right-hand factor $p$ of $(x-a)(x-b)$, and we may assume that $p$ is monic. We cannot have $p = 1$, since then $M = D[x,y]$. If $p$ is a linear factor $x-v$, then $v$ is a zero of $(x-a)(x-b)$, which again by our assumptions means that $vc \neq cv$ and hence  $D[x,y](x-v)+D[x,y](y-c) = D[x,y]$. But $D[x,y](x-v)+D[x,y](y-c) = D[x,y]p+D[x,y](y-c) \subseteq M$, hence $M = D[x,y]$, a contradiction. Thus $p$ cannot be linear, and hence we must have $p = (x-a)(x-b)$. But this means that $M \cap D[x] = D[x]p= D[x](x-a)(x-b)$ is not a maximal ideal in $D[x]$, since clearly it is strictly contained in $D[x](x-b)$. \end{proof}

Given a field $K$ and an automorphism $\sigma$ of $K$, let $K((t,\sigma))$ denote the corresponding skew Laurent power series ring over $K$, whose elements are Laurent power series in the variable $t$, and multiplication is determined by the rule $tu = u^\sigma t$ for all $u \in K$ (the so-called ``Hilbert twist"). It is well-known that such a ring $K((t,\sigma))$ is a division ring \cite[p.~23]{Krylov}.






\begin{proposition}\label{prop:example} Let $K$ be a field, equipped with an automorphism $\sigma$, and let $c$ be an element of $K$ such that $c^\sigma \neq c$ and $c^{\sigma^2}=c$. Let $D = K((t,\sigma))$ and let $M$ be the left ideal in $D[x,y]$ generated by $(x+t)(x-t),y-c$. Then $M$ is a maximal left ideal in $D[x,y]$ such that $D[x] \cap M$ is not maximal in $D[x]$.  
\end{proposition} \begin{proof}
Let $a = -t$, $b =t$. Then $$(ab)c = -t^2c = -c^{\sigma^2} t^2 = -ct^2 = c(ab)$$ and $c\cdot (a+b) = c\cdot 0 = 0 = (a+b)\cdot c$. 
  
Suppose now that $f = \sum_{i=m}^\infty f_i t^i \in D$ is a zero of $(x-a)(x-b) = (x+t)(x-t) = x^2-t^2$ (with $m \in \Z$, $f_m \neq 0$). Then from $f^2 = t^2$ we get, by comparing coefficients, that $m = 1$, $f_1 \neq 0$. We claim that $f$ does not commute with $c$. This follows by comparing the coefficient of $t$ in $fc$ and $cf$: In $fc$ this coefficient is $c^\sigma f_1$ while in $cf$ this coefficient is $cf_1$, and $c^\sigma f_1 \neq cf_1$ since $f_1 \neq 0$ and $c^\sigma \neq c$.

We have thus established all of the conditions of Proposition \ref{main}, from which the claim follows.
\end{proof}

Using Proposition \ref{prop:example} one can generate various concrete negative examples for the Amitsur-Small problem. For example, take $K$ as the algebraic closure $\bar{\F}_2$ of the field $\F_2$ of two elements, let $\sigma$ denote the Frobenius automorphism $a \mapsto a^2$ of $\bar{\F}_2$ and let $c \neq 1$ be an element of $K$ satisfying $c^3 = 1$. Then $c^{\sigma^2} = c^4 = c$ and $c^\sigma = c^2 \neq c$. Or, let $K$ be the field of complex numbers, let $\sigma$ denote complex conjugation, and take $c = \mathrm{i} = \sqrt{-1}$. Note that in the first example, since the Frobenius automorphism is of infinite order, the ring $D$ is of infinite dimension over its center $\F_2$, while in the second example the ring $D$ is centrally finite -- it is a quaternion algebra over its center $\R((t^2))$, with basis $1,\mathrm{i},t,\mathrm{i}t$.

\section{On the Nullstellensatz for algebraically closed division rings}

Let $K$ be an algebraically closed field. The classical Nullstellensatz, in its concrete form, describes the maximal ideals in the polynomial ring $K[x_1,\ldots,x_n]$: Those are precisely the ideals of the form $\langle x-a_1,\ldots,x-a_n \rangle$, for any point $(a_1,\ldots,a_n)$ in $K^n$. Now let $D$ be a division ring, and let us say that $D$ is {\it algebraically closed}, if every polynomial in $D[x]$ has a zero in $D$. It is then natural to ask -- what are the maximal left ideals in $D[x_1,\ldots,x_n]$, if $D$ is algebraically closed?

In the case where $D = \HH$ is the real quaternion algebra, a theorem of Niven and Jacobson \cite[Theorem 1]{Niv41} states that $D$ is algebraically closed, and then we have the concrete Nullstellensatz given in \cite{AP20}: The maximal left ideals in $D[x_1,\ldots,x_n]$ are precisely those of the form $\langle x-a_1,\ldots,x-a_n \rangle$, for points $(a_1,\ldots,a_n) \in D_c^n$. Some variations and refinements of this result are given in \cite{Aryapoor2024}, \cite{GSV24}, \cite{AP24}. These works are all focused on the real quaternion algebra $\HH$, and their techniques of proof heavily rely upon the explicit presentation of this ring.

For general division rings, we introduce the following terminology:

\begin{definition} Let $D$ be a division ring. We shall say that $D$ is a {\it Nullstellensatz ring}, if the maximal left ideals in $D[x_1,\ldots,x_n]$ are those of the form $\langle x-a_1,\ldots,x-a_n \rangle$, for points $(a_1,\ldots,a_n) \in D_c^n$.
\end{definition}

In other words, a Nullstellensatz ring is a ring for which the natural correspondence between maximal left ideals and points in $D_c^n$ holds, and the quaternion algebra $\HH$ is an example for such a ring, as mentioned above. Clearly, every Nullstellensatz ring must be algebraically closed (since an irreducible polynomial $p \in D[x]$ generates a maximal left ideal $D[x]p$), but we do not know whether the converse holds. However, this problem can be rephrased in terms of the Amitsur-Small problem discussed in this paper: 

\begin{definition}
     Let $D$ be a division ring. We shall say that $D$ is an {\it Amitsur-Small ring}, if for every $1 \leq k \leq n$ and for every maximal left ideal $M$ in $D[x_1,\ldots,x_n]$, the intersection $M \cap D[x_1,\ldots,x_k]$ is a maximal left ideal in $D[x_1,\ldots,x_k]$.
\end{definition}

We then have:

\begin{proposition}\label{null}
    Let $D$ be an algebraically closed division ring. The following are equivalent:
    \begin{enumerate}
        \item The ring $D$ is a Nullstellensatz ring.
        \item  The ring $D$ is an Amitsur-Small ring.
    \end{enumerate}
\end{proposition}
\begin{proof}
Suppose that $D$ is a Nullstellensatz ring, and let $M = \langle x_1-a_1,\ldots,x_n-a_n \rangle$ be a maximal left ideal in $D[x_1,\ldots,x_n]$ for a suitable point $(a_1,\ldots,a_n) \in D_c^n$. Then for each $1 \leq k \leq n$, we have $x_1-a_1,\ldots,x_k-a_k \in M \cap D[x_1,\ldots,x_k]$. But $\langle x_1-a_1,\ldots,x_k-a_k \rangle$ is a maximal left ideal in $D[x_1,\ldots,x_k]$, by Lemma \ref{whole}, hence we must have equality $M \cap D[x_1,\ldots,x_k] = \langle x_1-a_1,\ldots,x_k-a_k \rangle$, hence $M \cap D[x_1,\ldots,x_k]$ is maximal. Thus $D$ is an Amitsur-Small ring.

Conversely, suppose that $D$ is an Amitsur-Small ring, and let $M$ be a maximal left ideal in $D[x_1,\ldots,x_n]$. By \cite[Lemma A]{AmitsurSmall1978}, for each $1 \leq i \leq n$ we have $M \cap D[x_i] \neq 0$, and since $D[x_i]$ is a left principal ideal domain, we have $M \cap D[x_i] = D[x_i]p_i$, where $p_i$ is a monic polynomial of minimal degree in $x_i$ in $M$. By our assumptions, $D[x_i]p_i$ is a maximal left ideal in $D[x_i]$, hence $p_i$ must be irreducible in $D[x_i]$. But since $D$ is algebraically closed, $p_i$ is right-hand divisible by $x_i-a_i$ for some $a_i \in D$, hence $p_i = x_i-a_i$. Thus $M$ contains the left ideal $\langle x_1-a_1,\ldots,x_n-a_n\rangle$. If $(a_1,\ldots,a_n) \notin D_c^n$ we get, by Lemma \ref{whole}, that $M = D[x_1,\ldots,x_n]$, a contradiction. Thus $(a_1,\ldots,a_n) \in D_c^n$ and hence by Lemma \ref{whole} $\langle x_1-a_1,\ldots,x_n-a_n\rangle$ is a maximal left ideal, hence $M = \langle x_1-a_1,\ldots,x_n-a_n\rangle$. Thus $D$ is a Nullstellensatz ring. 
\end{proof}

As an immediate consequence of Proposition \ref{null}, we deduce that the real quaternion algebra $\HH$ is an example of an Amitsur-Small ring, since it is a Nullstellensatz ring, as discussed above. There are other known examples of algebraically closed non-commutative division rings, the first due to Makar-Limanov in \cite{Makar85}, and a variation of it in \cite{Kolesnikov2000}. These rings satisfy an even stronger property -- every polynomial function over them admits a zero. We do not know whether these rings are Nullstellensatz/Amitsur-Small rings.

\begin{remark}
Since $\HH$ is an Amitsur-Small ring, we deduce that the conditions of Proposition \ref{main} cannot be met with $D = \HH$. It is instructive to see this via a direct argument: Suppose that $a,b,c\in \HH$ satisfy the conditions of Proposition \ref{main}. Then in particular, $c \notin \R$, since $c$ must not commute with the zero $b$ of $p = (x-a)(x-b) = x^2-(a+b)x+ab$. Now, since $ab$ commutes with $c$, $ab$ belongs to the subfield $\R(c)$ of $\HH$. Similarly we have $a+b \in \R(c)$. Thus $p \in \R(c)[x]$. But the subfield $\R(c)$ of $\HH$ is necessarily isomorphic to $\C$, hence there exists a zero of $p$ in $\R(c)$, which of course commutes with $c$, a contradiction.
\end{remark}

We conclude this note with the following questions:
\begin{enumerate}
    \item Is there a simple, general characterization of the Amitsur-Small rings?
    \item  Can one characterize the quaternion algebras that are Amitsur-Small rings? 
    \item Are all algebraically closed division rings Amitsur-Small rings? (Equivalently, are all algebraically closed division rings Nullstellensatz rings?)

\end{enumerate}



\end{document}